\documentclass{article}
\usepackage{graphicx} 
\usepackage[a4paper, left=3cm, right=3cm, top=2cm]{geometry}
\usepackage{amsmath, amssymb, amsthm, mathtools}
\usepackage[english]{babel}
\usepackage{tikz-cd}
\usepackage{xcolor}
\usepackage{etoolbox}
\usepackage{enumitem}
\usepackage{scrextend}
\usepackage[parfill]{parskip}
\usepackage{titlesec}
\usepackage{csquotes}
\usepackage{abstract}
\usepackage{xfakebold}
\usepackage{comment}
\usepackage{breqn}
\usepackage{hyperref}
\hypersetup{
  colorlinks=true,
  allcolors=blue
}

\theoremstyle{definition}

\newtheorem{lemma}{Lemma}
\newtheorem{proposition}[lemma]{Proposition}
\newtheorem*{definition}{Definition}
\newtheorem*{theorem*}{Theorem}

\let\originaliota\iota
\renewcommand{\iota}{\dot\originaliota}

\AtBeginEnvironment{theorem}{\setlist{before=\leavevmode, nosep}}

\titleformat{\section}
  {\Large\color{black}}
  {\thesection}{1em}{}
\titleformat{\subsection}
  {\Large\color{black}}
  {\thesubsection}{1em}{}
\titleformat{\subsubsection}
  {\color{black}}
  {\thesubsubsection}{1em}{}

\newcommand{\monthyear}{%
  \ifcase \month \or January\or February\or March\or April\or May\or June%
  \or July\or August\or September\or October\or November\or December\fi, \number \year
} 
\newcommand{\C}{\mathbb{C}}
\newcommand{\R}{\mathbb{R}}

\begin{document}

\title{Geometric quantization of generalized Hirzebruch fibrations}
\author{Andrea Piccirilli
\thanks{I would be grateful for potential comments and corrections.
\newline
Email address: \href{mailto:apiccirilli@ethz.ch}{apiccirilli@ethz.ch}.}}
\date{}
\maketitle

\begin{abstract}
 Hirzebruch surfaces, defined as the projectivization of line bundles over $\C\mathbb{P}^1$, support a toric action and thus represent an infinite class of symplectic toric manifolds of complex dimension 2. In this paper, an infinite class of toric manifolds given as projective bundles over $\mathbb{C}\mathbb{P}^d$ will be constructed for every complex dimension $d$ and it will be shown that each manifold supports a symplectic structure. With the toric and symplectic structure of the manifolds at our disposal, we then study their geometric quantization and how it relates to different values of the twisting parameter of the fibrations.
\end{abstract}
\section{Introduction}
In this paper, we present an explicit construction of a class of manifolds that we call 'generalized Hirzebruch fibrations'. These varieties appear in the literature and are known to be toric: in fact, any line bundle over a toric variety carries a natural toric structure \cite{Audin}, however, their symplectic aspects are typically treated only implicitly. Here, we give a direct symplectic approach by describing these varieties explicitly as the zero set of certain homogeneous polynomials. This coordinate description allows us to define an effective toric action and compute the corresponding moment polytope.
\\
The main theorem we will use relates the dimension of the space of holomorphic sections (i.e., the geometric quantization with respect to a Kähler polarization) to the number of integer lattice points in the moment polytope \cite{Mard}:
\\
\begin{theorem*}
\label{main}
\textit{Let $(M^{2n},\omega)$ be a symplectic toric manifold with associated moment polytope $\Delta \subset \mathbb{R}^n$ with integer vertices. The dimension of the quantization space (that is, the number of independent holomorphic sections of the line bundle $\mathbb{L}$) is equal to the number of integer lattice points in $\Delta$:}
\begin{equation}
   \dim \mathcal{Q}(M,\omega) = \#(\Delta \cap \mathbb{Z}^{n}).
    \end{equation}
\end{theorem*}
The condition of integer vertices is a necessary and sufficient condition for the geometric quantization of symplectic toric manifolds. The quantization space is a complex vector space of holomorphic functions in $n$ variables $z_1,...,z_2$ spanned by 
\begin{equation}
    \mathcal{B}:=\Big\{z^{j_1}_1...z^{j_n}_n\:|\:\{j_1,...,j_n\}\in \Delta \cap \mathbb{Z}^{n}\Big\}.
\end{equation}
By using this result and our symplectic toric description of the generalized Hirzebruch fibration,
we derive an exact formula for their quantization dimension as a function of the twisting parameter. We then show that this function satisfies a linear recurrence relation, and we analyze its asymptotic behavior in relation to the symplectic volume of the manifold.
\\
\\
In Section 1, we illustrate how quantization for symplectic toric manifolds works by studying complex projective space. This will serve not only as an exposition of the methods in the paper, but the result obtained here will be needed in the later sections. Section 2 contains the definition of generalized Hirzebruch fibrations and a detailed construction of their toric action and moment polytope
In Section 3, we study the behavior of the quantization for these fibrations: in particular, we obtain a recursion relation and study the limit where the Chern classes of the fibration becomes large. 
\\
\\

\section{Primer: quantization of $\C\mathbb{P}^N$}
In this section, we compute the quantization for complex projective space. This not only illustrates the methods in a simpler setting than generalized Hirzebruch fibrations, but, given that these fibrations are over $\C\mathbb{P}^d$, the result we obtain here will be needed in the next sections.
\\
\\
We take as symplectic form $\omega_b := b\omega_{FS}$ for a $b \in \mathbb{Z}$, where, for homogeneous coordinates $[z_0:...:z_N]$, we take the following Fubini-Study form:
\begin{equation*}
    \omega_{FS} = i\partial\overline{\partial}\log\bigg(\sum_{i = 0}^{N}|z_i|^2\bigg).
\end{equation*}
With this normalization, the volume of $\C\mathbb{P}^N$ is 
\begin{equation*}
    \text{Vol}(\C\mathbb{P}^N) = \int_{\C\mathbb{P}^N}\omega^N_{FS} = (2\pi)^n,
\end{equation*}
that is, we have a factor $2^n$ of difference from the usual form. This is to ensure that all vertices of the moment polytope have integer coordinates, making the quantization well-defined. 
\\
\\
The natural action $\mathbb{T}^{N}$ on $\C\mathbb{P}^N$ is 
\begin{equation}
    (\theta_1,...,\theta_N) \cdot [z_0:...:z_n] := [z_0:e^{i\theta_1}z_1:...:e^{i\theta_n}z_n],
\end{equation}
which is effective. The moment map $\mu: \: \C\mathbb{P}^N \to \text{Lie}(\mathbb{T}^{N})^* \cong \mathbb{R}^N$ is given by
\begin{equation*}
    \mu([z_0 : ... :z_N]) = b\bigg(\frac{|z_1|^2}{|z_0|^2+|z_1|^2+...+|z_N|^2},...,\frac{|z_N|^2}{|z_0|^2+|z_1|^2...+|z_N|^2}\bigg).
\end{equation*}
\\
The fixed points of $\mu$ are $[1:0:...:0],...,[0:...:0:1]$.
The associated moment polytope is the following simplex:
\begin{equation*}
    \Delta^{N}_b := \bigg\{\sum_{i=1}^{N} x_i \leq b \: \bigg| \: 0\leq x_i\leq b \quad \forall\, i = 1,...,N \bigg\}.
\end{equation*}
We want to use the main theorem to compute the dimension of the associated quantization space by counting integer lattice points in $\Delta^{N}_b$.
\\
Let $q = 0,...,b$. The integer lattice points correspond exactly to the number of solutions to the following Diophantine equation:
\begin{equation*}
    \sum_{i=1}^{N} k_i = q, \quad k_i \in \mathbb{Z}.
\end{equation*}
This number is exactly $\binom{q+N-1}{N-1}$. Therefore, the total number of integer points is given by 
\begin{equation*}
  \#(\Delta \cap \mathbb{Z}^N) = \sum_{q=0}^{b} \binom{q+N-1}{N-1}
\end{equation*}
\begin{lemma} \textit{1}
 \begin{equation*}
     \sum_{q=0}^{b} \binom{q+N-1}{N-1} = \binom{b+N}{N}.
 \end{equation*}
\end{lemma}
\begin{proof}
    We prove the proposition by induction on $b$:
    \begin{itemize}
        \item $b=0$: both sides are equal to 1
        \item Induction step: $b \to b+1$: 
          \begin{equation*}
    \begin{split}
      \sum_{q=0}^{b+1} \binom{q+N-1}{N-1} & = \sum_{q=0}^{b} \binom{q+N-1}{N-1}+
      \binom{b+N}{N-1}\\
      & = \binom{q+N-1}{N-1} +\binom{b+N}{N-1}\\
      & = \binom{b+1+N}{N},
  \end{split}
  \end{equation*} 
    \end{itemize}
    proving the formula above.
\end{proof}
\noindent
\\
By using the main theorem, we obtain the result:
    \begin{equation}
        \dim (\mathcal{Q}(\C\mathbb{P}^N, \omega_b)) = \binom{b+N}{N}.
\end{equation}
\section{Generalizations Hirzebruch fibrations}
To begin, we consider the following projectivization of line bundles over $\C\mathbb{P}^d$:
\begin{equation*}
    \pi  : \:\text{Proj}(\mathcal{O}(-m)\oplus \mathcal{O}(-n)) \to \C\mathbb{P}^d.
\end{equation*}
Notice that, because of the projectivization, the following chain of isomorphisms holds:
\begin{equation*}
    \C\mathbb{P}(\mathcal{O}(-m)\oplus \mathcal{O}(-n))  \cong \C\mathbb{P}((\mathcal{O}(-m)\oplus \mathcal{O}(-n))\otimes \mathcal{O}(m))
    \cong \C\mathbb{P}(\mathcal{O}\oplus \mathcal{O}(m-n)).
\end{equation*}
Therefore, we can always assume that one of the two integers is zero and the other is nonnegative: we set $m=0$ and $n \in \mathbb{N}\cup\{0\}$. This leads us to the next and main definition.
\\
\begin{definition}
    \textit{Let $n\in n \in \mathbb{N}\cup\{0\}$. The \textbf{$n$-th generalized Hirzebruch fibration} is defined to be the projectivization of the following bundle over $\C \mathbb{P}^d$: }
    \begin{equation}
        \mathcal{H}^d_n := \text{Proj}(\mathcal{O}\oplus \mathcal{O}(-n)) \to \C\mathbb{P}^d,
    \end{equation}
   \textit{ where $\mathcal{O}(n):= \bigoplus_n \mathcal{O}(1)$ is the $n$-th twisting line bundle of $\C \mathbb{P}^d$ and $\mathcal{O}(-n) : = \mathcal{O}(n)^*$.}
   \\
\end{definition}
\noindent
The parameter $n$ is called \textbf{twisting parameter} of the fibration.
Note that $\mathcal{H}^1_n$ are the familiar Hirzebruch surfaces. We wish to understand this fibration more in detail algebraically.
\\
\\
Let $\alpha := c_1(\mathcal{O}(1))\in H^2(\C \mathbb{P}^d, \mathbb{Z})$ the Chern class of the first twisting line bundle. Moreover, the total Chern class of $\mathcal{O}\oplus \mathcal{O}(-n)$ is 
\begin{equation*}
  c(\mathcal{O}\oplus \mathcal{O}(-n)) = 1-n\alpha  
\end{equation*}
because of the splitting principle. There is a classic result for the cohomology ring of a projectivization $\pi:\:\text{Proj}(E)\to B$ of a vector bundle $E$ of rank $r$: 
\begin{equation*}
    H^*(\text{Proj}(E),\mathbb{Z}) = H^*(B, \mathbb{Z})[\xi]\Big/\Big\langle \xi^r + \pi^*c_1(E)\xi^{r-1}+...+\pi^*c_1(E)\Big\rangle,
\end{equation*}
where $\xi:= c_1(\mathcal{O}_{\text{Proj}(E)}(1))$ is the first Chern class of the tautological line bundle on $\text{Proj}(E)$. The total Chern class is 
\begin{equation*}
    c(\text{Proj}(E)) = \pi^*c(B)\cdot\frac{c(E)}{1+\xi}.
\end{equation*}
For $E = \mathcal{O}\oplus \mathcal{O}(-n)$, the rank is 2. The total Chern class of $\C\mathbb{P}^d$ is $c(\C\mathbb{P}^d) = (1+\alpha)^{d+1}$, therefore the total Chern class of the generalized Hirzebruch fibration is
\begin{equation}
    c(\mathcal{H}^d_n) = \pi^*(1+\alpha)^{d+1}\cdot\frac{1-n\pi^*\alpha}{1+\xi}, 
\end{equation}
where $\xi$ satisfies the relation $\xi^2+n\pi^*\alpha\xi = 0$ in the cohomology ring
\begin{equation*}
    H^*(\mathcal{H}^d_n,\mathbb{Z}) \cong \mathbb{Z}[\alpha,\xi]\Big/\Big\langle \alpha^{d+1}, \xi^2+n\pi^*\alpha\xi \Big\rangle.
\end{equation*}
In particular, the first Chern class is 
\begin{equation*}
    c_1(\mathcal{H}^d_n) = (d+1-n)\pi^*\alpha-\xi.
\end{equation*}
The goal is to put a toric structure on the generalized Hirzebruch manifolds. As it is usually done for the case of Hirzebruch surfaces, we want to embed the abstract surface in a product of complex projective spaces and compute the moment map by finding a suitable restriction on the torus action of the ambient space.
\\
\begin{proposition}
   \textit{ The map}
\begin{equation}
    \begin{split}
  \phi  :\:   \mathcal{H}^d_n & \to \C\mathbb{P}^d\times\C\mathbb{P}^{d+1}\\
   ([x_0:...:x_d],[z_0:z_1]) & \mapsto ([x_0:...:x_d],[z_0:z_1x^n_0:z_1x^n_1:...:z_1x^n_d]).
\end{split}
\end{equation}
 \textit{is an embedding of complex manifolds. }
 \\
\end{proposition}
\begin{proof}
We proceed in several steps.
\\
\\
A point in $\mathcal{H}_n^d$ is given by a pair 
\[
\bigl( [x_0 : \cdots : x_d], [z_0 : z_1] \bigr),
\]
where the coordinates are defined up to the usual projective equivalence. Since $\phi$ is defined by the homogeneous functions
\[
z_0 \quad \text{and} \quad z_1 x_i^n \quad (i=0,\dots,d),
\]
the map is independent of the choice of representatives. Moreover, as these functions are holomorphic, $\phi$ is holomorphic.
\\
We now show that it's injective. 
Suppose 
\[
\phi\Bigl( [x_0 : \cdots : x_d], [z_0 : z_1] \Bigr) = \phi\Bigl( [x'_0 : ... : x'_d], [z'_0 : z'_1] \Bigr).
\]
Then, by the definition of $\phi$, we have
\[
[x_0 : \cdots : x_d] = [\widetilde{x}_0 : \cdots : \widetilde{x}_d] \quad \text{in } \mathbb{C}\mathbb{P}^d.
\]
Next, comparing the second component we have
\begin{equation*}
 [z_0 : z_1 x_0^n : \cdots : z_1 x_d^n] = [\widetilde{z}_0 : \widetilde{z}_1 \widetilde{x}_0^n : ... : \widetilde{z}_1 \widetilde{x}_d^n].   
\end{equation*}
Replacing $\widetilde{x}_i$ by $\lambda x_i$, the fiber coordinates scale as $[z_0:\lambda^{-n}z_1]$, therefore
\begin{equation*}
  [\widetilde{z}_0 : \widetilde{z}_1 \widetilde{x}_0^n : ... : \widetilde{z}_1 \widetilde{x}_d^n]
= [z'_0 : \lambda^n z'_1 x_0^n : \cdots : \lambda^n z'_1 x_d^n]
\end{equation*}
Thus, there exists $\mu \in \mathbb{C}^*$ such that
\[
z'_0 = \mu z_0 \quad \text{and} \quad \lambda^n z'_1 = \mu z_1.
\]
It follows that 
\[
[z'_0 : z'_1] = [z_0 : z_1],
\]
so $\phi$ is injective. Moreover, since $\phi$ is defined by homogeneous polynomials on projective coordinates, a local computation in any standard coordinate chart shows that the differential of $\phi$ is injective at every point. Therefore, $\phi$ is an immersion.
\\
Finally, since $\mathcal{H}_n^d$ is compact and $\mathbb{C}\mathbb{P}^d \times \mathbb{C}\mathbb{P}^{d+1}$ is Hausdorff, any holomorphic, injective immersion from a compact complex manifold into a Hausdorff space is a homeomorphism onto its image. Thus, $\phi$ is an embedding.

\end{proof}
\noindent
The manifolds $\mathcal{H}^d_n$ can also be described as the zero locus of homogeneous polynomials of degree $n+1$, which gives them the structure of projective varieties. Let 
\begin{equation*}
    ([x_0:...:x_d],[z_0:y_0:...:y_d]) \in \text{im}(\phi)\subset\C\mathbb{P}^d\times\C\mathbb{P}^{d+1}.
\end{equation*}
By the embedding, we obtain 
\begin{equation*}
    \frac{y_i}{y_0} = \frac{x^n_i}{x^n_0} \Longleftrightarrow y_i x^n_0 - y_0 x^n_i = 0 \quad \forall \: i \in \{1,...,d\}.
\end{equation*}
The $d$ equations above are all independent, while for $i=0$ the equation is trivially satisfied. The other coordinates satisfy the same relations with respective indices, however, these equations are not all functionally independent. The equations in terms of $y_0$ imply all the other relations between the coordinates:
\begin{equation*}
    \frac{y_i}{y_j} = \frac{y_i}{y_0}\frac{y_0}{y_j} = \frac{x^n_i}{x^n_0} \frac{x^n_0}{x^n_j} = \frac{x^n_i}{x^n_j} \Longleftrightarrow y_i x^n_j - y_j x^n_i = 0 \quad \forall \: i,j \in \{1,...,d\}.
\end{equation*}
We write 
\begin{equation}
    \mathcal{H}^d_n = Z(y_0 x^n_i - y_i x^n_0\:|\: i \in \{1,...,d\}).
\end{equation}
With this presentation, we can immediately compute the dimension of the manifolds:
\begin{equation}
    \dim_{\C} \mathcal{H}^d_n = \dim_{\C}(\C\mathbb{P}^d\times\C\mathbb{P}^{d+1})-d = d+1.
\end{equation}
We will now endow the $\mathcal{H}^d_n$ with the structure of a toric manifold, specifically, we will define an action of $\mathbb{T}^{d+1}$ on $\C\mathbb{P}^d\times\C\mathbb{P}^{d+1}$ as follows:
\begin{equation}
\begin{split}
     &(\theta_1,...,\theta_{d+1})\: \cdot \: \big([x_0:x_1:...:x_d],[z_0:y_0:y_1...:y_d]\big) :=\\
     & = \big([x_0:e^{i\theta_1}x_1:...:e^{i\theta_d}x_d],[z_0:e^{i\theta_{d+1}}y_0:e^{i(n\theta_1+\theta_{d+1})}y_1:...:e^{i(n\theta_d+\theta_{d+1})}y_d]\big).
\end{split}
\end{equation}
The action is effective, which makes $\mathcal{H}^d_n$ a toric manifold.
The fixed points of this action are of the form
\begin{equation*}
    ([0:...:\underbrace{1}_{i-\text{th}}:...:0],[0:...:\underbrace{1}_{j-\text{th}}:...:0]\big), 
\end{equation*}
and there are exactly $(d+1)(d+2)$ of them. 
\\
\\
The Kähler form for the product of projective spaces is given as a linear combination  of the respective Fubini-Study forms (with the same normalization as the first section), i.e. 
\begin{equation}
\begin{split}
     \omega_{a,b}  & := \omega_a + \omega_{b}\\
     & = -ia\partial\overline{\partial}\log\bigg(\sum_{i = 0}^{d}|x_i|^2\bigg)-ib\partial\overline{\partial}\log\bigg(\sum_{i = -1}^{d}|y_i|^2\bigg)\\
     & = -ia \sum_{i,j=0}^d \left( \frac{\delta_{ij}}{\sum_{i=0}^d |x_i|^2} - \frac{x_i \bar{x}_j}{\left( \sum_{i=0}^d |x_i|^2 \right)^2} \right) dx_i \wedge d\bar{x}_j
+\\
& -ib\sum_{l,m=-1}^{d} \left( \frac{\delta_{lm}}{\sum_{j=-1}^{d} |y_j|^2} - \frac{y_l \bar{y}_m}{\left( \sum_{j=-1}^{d} |y_j|^2 \right)^2} \right) dy_l \wedge d\bar{y}_m.
\end{split}
\end{equation}
where $a,b \in \mathbb{Z}$. Note that these two parameters only have an impact on the symplectic structure, but not on the algebraic and complex geometry of the fibrations. Equipped with this symplectic form, the generalized Hirzebruch manifolds become symplectic manifolds, which we denote as $(\mathcal{H}^d_{a,b,n}, \omega_{a,b})$.
\\
\\
We now proceed to show that the $\mathbb{T}^{d+1}$-action defined above is Hamiltonian. Consider $\theta_i$ for $1\leq i \leq d$. The vector fields $\frac{\partial}{\partial \theta_i}$ generating the rotational action in the $\theta_i$-direction can be written in terms of complex projective coordinates: 
\begin{equation*}
\begin{split}
    \frac{\partial}{\partial \theta_i} &  = \frac{\partial x_i}{\partial \theta_i} \frac{\partial }{\partial x_i}+\frac{\partial \bar{x}_i}{\partial \theta_i} \frac{\partial }{\partial \bar{x}_i}+\frac{\partial y_i}{\partial \theta_i} \frac{\partial }{\partial y_i}+\frac{\partial \bar{y}_i}{\partial \theta_i} \frac{\partial }{\partial \bar{y}_i}\\
    & = i \bigg(x_i\frac{\partial }{\partial x_i}- \bar{x}_i\frac{\partial }{\partial \bar{x}_i}\bigg)+ i n \bigg(y_i\frac{\partial }{\partial y_i}- \bar{y}_i\frac{\partial }{\partial \bar{y}_i}\bigg).
\end{split}
\end{equation*}
For the last vector field, we obtain 
\begin{equation*}
    \frac{\partial}{\partial \theta_{d+1}}  = \sum_{j = 0}^{d}\bigg(\frac{\partial y_j}{\partial \theta_{d+1}} \frac{\partial }{\partial y_j}+\frac{\partial \bar{y}_j}{\partial \theta_{d+1}} \frac{\partial }{\partial \bar{y}_{j}}\bigg)
    = i \sum_{j = 0}^{d}\bigg(y_j\frac{\partial }{\partial y_j}- \bar{y}_j\frac{\partial }{\partial \bar{y}_j}\bigg).
\end{equation*}
Computing the insertion of these vector fields in the Kähler form, we obtain 
\begin{equation*}
\begin{split}
     \iota_{\frac{\partial}{\partial \theta_p}}\omega & = -\frac{1}{2} \sum_{i,j=0}^d \left( \frac{\delta_{ij}}{\sum_{i=0}^d |x_i|^2} - \frac{x_i \bar{x}_j}{\left( \sum_{i=0}^d |x_i|^2 \right)^2} \right) \left( x_p \delta_{ip} \, dx_p \wedge d\bar{x}_p - \bar{x}_p \delta_{jp} \, dx_i \wedge d\bar{x}_j \right)+\\
& - \frac{n}{2} \sum_{l,m=-1}^d \left( \frac{\delta_{lm}}{\sum_{j=-1}^d |y_j|^2} - \frac{y_l \bar{y}_m}{\left( \sum_{j=-1}^d |y_j|^2 \right)^2} \right) \left( y_p \delta_{lp} \, dy_p \wedge d\bar{y}_p -  \bar{y}_p \delta_{mp} \, dy_l \wedge d\bar{y}_m \right).
\end{split}
\end{equation*}
We obtain the following moment map $\mu: \: \C\mathbb{P}^d\times\C\mathbb{P}^{d+1}\to \text{Lie}(\mathbb{T}^{d+1})^* \cong \R^{d+1}$:
\begin{equation}
    \mu([x_0:x_1:...:x_d],[z_0:y_0:y_1...:y_d]) = 
    \begin{pmatrix}
        a\frac{|x_1|^2}{\sum_{i=0}^d |x_i|^2}+nb\frac{|y_1|^2}{|z_0|^2+\sum_{i=0}^d |y_i|^2}\\
        \vdots\\
        a\frac{|x_d|^2}{\sum_{i=0}^d |x_i|^2}+nb\frac{|y_d|^2}{|z_0|^2+\sum_{i=0}^d |y_i|^2}\\
        b\frac{\sum_{i=0}^d |y_i|^2}{|z_0|^2+\sum_{i=0}^d |y_i|^2}\\
    \end{pmatrix}
\end{equation}
Now that we have shown that the action of $\mathbb{T}^{d+1}$ is toric, we should restrict the moment map to $\mathcal{H}^d_n$, which in local coordinates where $x_0 = 1$, corresponds to setting $y_i = y_0 x^n_i$. 
\begin{equation}
    \mu([1:x_1:...:x_d],[z_0:y_0:y_0x^n_1...:y_0x^n_d]) = 
    \begin{pmatrix}
        a\frac{|x_1|^2}{\sum_{i=0}^d |x_i|^2}+nb\frac{|y_0|^2|x_1|^{2n}}{|z_0|^2+|y_0|^2(1+\sum_{i=0}^d |x_i|^{2n})}\\
        \vdots\\
        a\frac{|x_d|^2}{\sum_{i=0}^d |x_i|^2}+nb\frac{|y_0|^2|x_d|^{2n}}{|z_0|^2+|y_0|^2(1+\sum_{i=0}^d |x_i|^{2n})}\\
        b\frac{\sum_{i=0}^d |y_i|^2}{|z_0|^2+\sum_{i=0}^d |y_i|^2}\\
    \end{pmatrix}
\end{equation}
We are interested in understanding the combinatorics of the moment polytope $\mu(\mathcal{H}^d_{a,b,n})$ associated with the generalized Hirzebruch manifold.
First, we need to determine which fixed points of the ambient action are also fixed points of the action restricted to the generalized Hirzebruch manifold. 
\\
\\
Consider the collection of $d+1$ open sets 
\begin{equation*}
    U_i := \{[z_0:...:z_d]\in \C\mathbb{P}^d\:|\:z_i \neq 0\} \cong \C^d,
\end{equation*}
which form an atlas for $\C\mathbb{P}^d$. Given the form of the fixed points, as described above, we note that each $U_i$ contains $d+2$ of them, and $U_i \cap U_j$ does not contain a fixed point for $i\neq j$. Clearly, the collection of $V_i :=U_i \cap \mathcal{H}^d_n$, for $i=0,..,d$, forms an atlas for $\mathcal{H}^d_n$, and on each $V_i$ the equation $y_j = y_i x^n_j$ holds. A fixed point has $x_j = 0$ for $i\neq j$, therefore necessarily $y_j = 0$. We conclude that $V_i$ contains two fixed points, given by 
\begin{equation*}
     Z^{(0)}_i := ([0:...:\underbrace{1}_{i-\text{th}}:...:0],[1:...:0])
\end{equation*}
and 
\begin{equation*}
     Z^{(1)}_i := ([0:...:\underbrace{1}_{i-\text{th}}:...:0],[0:...:\underbrace{1}_{i-\text{th}}:...:0]\big).
\end{equation*}
Therefore, only $2d+2$ fixed points of the ambient action are also fixed points of the restricted action. The image of these points under the moment map is precisely the set of vertices of the moment polytope $\mu(\mathcal{H}^d_{a,b,n})$:
\begin{equation*}
    \mu(Z^{(0)}_0) = 0
\end{equation*}
\begin{equation*}
    \mu(Z^{(1)}_0) = be_{d+1}
\end{equation*}
For $ =1,...,d$ we obtain 
\begin{equation*}
    \mu(Z^{(0)}_i) = (a+nb)e_i
\end{equation*}
\begin{equation*}
    \mu(Z^{(1)}_i) = ae_i+be_{d+1}
\end{equation*}
The associated moment polytope to the fibration is obtained as the convex hull of these vertices: 
\begin{equation}
    \Delta^d_{a,b,n}:=\bigg\{(x_1,...,x_{d+1}) \in \R^{d+1}\:\bigg|\:x_i \geq0,\quad\sum_{i=i}^d x_i \leq a +n(b-x_{d+1}), \quad x_{d+1}\leq b\bigg\}
\end{equation}
\section{Properties of the quantization function}
We define the \textbf{quantization function} as $ \mathcal{Q}_{a,b,d}(n):=\mathcal{Q}(\mathcal{H}^d_{a,b,n},\omega_{a,b})$
The goal is to compute this function by counting integer lattice points for generalized Hirzebruch fibrations. To do so, we foliate the polytope in the direction normal to the basis vector $e_{d+1} \in \mathbb{R}^{d+1}$:
\begin{equation}
\label{foliation}
    \Delta^d_{a,b,n}(t) := \bigg\{(x_1,...,x_{d+1}) \in \R^{d+1}\:\bigg|\:x_i \geq0, \quad x_{d+1}= t, \quad\sum_{i=i}^d x_i \leq L_t := a +n(b-t)\bigg\}.
\end{equation}
Note that, for a fixed $t$, this is a standard $d$-simplex scaled by $L_t$. 
Clearly 
\begin{equation}
   \Delta^d_{a,b} = \bigsqcup_{t = 0}^{b}  \Delta^d_{a,b}(t).
\end{equation}
The number of $(l_1,...,l_{d})\in \mathbb{Z}^{d}_{\geq 0}$ satisfying the Diophantine inequality $\sum_{i=i}^d l_i \leq L_t$ is $\binom{L_t+d}{d}$, a result that can be obtained by the same method used in the first section to compute the quantization of complex projective space.
We obtain the quantization as a function of all the parameters of the generalized Hirzebruch fibration:
\begin{equation}
\begin{split}
    \mathcal{Q}_{a,b,d}(n) &:= \#(\Delta^d_{a,b} \cap \mathbb{Z}^{d+1})\\
    & = \sum_{t = 0}^{b} \#(\Delta^d_{a,b}(t)\cap \mathbb{Z}^{d+1})\\
    & = \sum_{t = 0}^{b} \binom{L_t+d}{d}\\
    & = \sum_{t = 0}^{b} \binom{a+n(b-t)+d}{d}\\
    & = \sum_{i = 0}^{b} \binom{a+d+ni}{d}.\\
    & = \mathcal{Q}(\C\mathbb{P}^d,\omega_a)+\sum_{i = 1}^{b} \binom{a+d+ni}{d}.
\end{split}
\end{equation}
The first contribution always comes from the quantization of $\C\mathbb{P}^d$, that is, the base of the bundle. The sum contains the contributions from the fibers of the bundle.
In general, this sum does not admit a solution in closed form. However, in the last section, we will derive an exact recursion relation for $\mathcal{Q}_{a,b,d}(n)$. For now, let us consider some special cases for which the sum can be evaluated explicitly. 
\\
\\
For $d = 1$, $\mathcal{H}^1_{a,b,n}$ are the usual Hirzebruch surfaces.. We obtain
\begin{equation}
\label{Hirzebruch}
    \mathcal{Q}_{a,b,1}(n) = \Big(a+1+\frac{nb}{2}\Big)(b+1).
\end{equation}
For a general complex dimension, the only cases for which the quantization function can be evaluated explicitly are the first two fibrations, corresponding to $n=0$ and $n=1$. We show that this is related to the fact that for these fibrations the geometry is particularly simple and can be related to the geometry of complex projective spaces.
\\
\\
In the first case, we have $\mathcal{H}^d_{a,b,0} \cong \C\mathbb{P}^d \times \C\mathbb{P}^1$, which can be seen by the fact that the vector bundle $\mathcal{O}\oplus\mathcal{O}$ is trivial of rank 2, therefore it must be isomorphic to $\C\mathbb{P}^d \times \C$. After taking the projectivization, we obtain the isomorphism. Therefore, the quantization factorizes:
\begin{equation} 
    \mathcal{Q}_{a,b,d}(0) =  \binom{a+d}{d}(b+1) = \mathcal{Q}(\C\mathbb{P}^d)\mathcal{Q}(\C\mathbb{P}^1).
\end{equation}
The $n=1$ case is well-known when $d=1$: ignoring momentarily the parameters $a,b$, which are related to the symplectic structure, $\mathcal{H}^1_{1}$ is the blow-up of $\C\mathbb{P}^2$ at a point. We show a similar statement for a general value of $d$:
\\
\begin{proposition}
    \textit{The generalized Hirzebruch fibration corresponding to $n=1$ is isomorphic to the blow-up of $\C\mathbb{P}^{d+1}$ at a copy of $\C\mathbb{P}^{d-1}$:}
    \begin{equation}
        \text{Proj}(\mathcal{O}\oplus \mathcal{O}(-1)) \cong \text{Bl}_{\C\mathbb{P}^{d-1}}(\C\mathbb{P}^{d+1})
    \end{equation}
\end{proposition}
\begin{proof}
    Let $[z_0:...:z_{d+1}]$ be homogeneous coordinates on $\C\mathbb{P}^{d+1}$ and consider the subspace $P:=\{[0:0:x_2:...:x_{d+1}]\} \cong \C\mathbb{P}^{d-1}$. The blow-up $X:= \text{Bl}_P(\C\mathbb{P}^{d+1})$ can be described as the following subvariety of $\C\mathbb{P}^{d+1} \times \C\mathbb{P}^1$:
    \begin{equation*}
    X =\{([z_0:z_1:...:z_{d+1}],[w_0:w_1]) \in \C\mathbb{P}^{d+1} \times \C\mathbb{P}^1\:|\:z_0w_1 = z_1w_0\}
    \end{equation*}
    There is a natural projection 
    \begin{equation*}
         \begin{split}
  \pi  :\:   X & \to \C\mathbb{P}^{d+1}\\
   ([z_0:z_1:...:z_{d+1}],[w_0:w_1]) & \mapsto  [z_0:z_1:...:z_{d+1}]
\end{split}
\end{equation*}
and the exceptional divisor is $E:=\pi^{-1}(P) = P \times \C\mathbb{P}^1$. It holds $X\setminus E \cong \C\mathbb{P}^{d+1}\setminus P \cong \C\mathbb{P}^{d-1}$, 
where the first isomorphism is given by $\pi|_E$.
\\
\\
The first generalized Hirzebruch fibration can be described as the following vanishing locus:
\begin{equation*}
    \mathcal{H}^d_1= Z(y_0 x_i - y_i x_0\:|\: i \in \{1,...,d\})\in \C\mathbb{P}^{d}\times \C\mathbb{P}^{d+1}
\end{equation*}
We construct a map between  $\mathcal{H}^d_1$ and $X$ and show that it is an isomorphism.
\begin{equation*}
\begin{split}
  \Psi  :\:  \mathcal{H}^d_1  & \to X\\
  ([x_0:...:x_d],[z_0:y_0:...:y_d]) & \mapsto  ([y_0:y_1:...:y_{d}:z_0],[x_0:x_1])
\end{split}
\end{equation*}
We will show that $\Psi$ is well-defined, holomorphic, injective, and surjective, so that it is a biholomorphism.

\textbf{(i) Well-definedness and holomorphicity.}  
A point in $\mathcal{H}^d_1$ is represented by a pair
\[
\bigl( [x_0:x_1:\cdots:x_d],\, [z_0:y_0:y_1:\cdots:y_d] \bigr)
\]
where the second factor represents the fiber of the projectivization of $\mathcal{O}\oplus \mathcal{O}(-1)$.
The map 
\[
\Psi\Bigl( [x],[z:y] \Bigr) = \Bigl( [y_0:y_1:\cdots:y_d:z_0],\, [x_0:x_1] \Bigr)
\]
is defined by homogeneous polynomials in the coordinates. Hence it is independent of the chosen representatives and is holomorphic.
\\
\\
By the definition of $X$, a point $([z_0:z_1:\cdots:z_{d+1}], [w_0:w_1]) \in \mathbb{C}\mathbb{P}^{d+1}\times\mathbb{C}\mathbb{P}^1$ lies in $X$ if and only if 
\[
z_0\,w_1 = z_1\,w_0.
\]
For a point $P \in \mathcal{H}^d_1$, its image under $\Psi$ is
\[
\Psi(P)=\Bigl( [y_0:y_1:\cdots:y_d:z_0],\, [x_0:x_1] \Bigr).
\]
By the defining equations of $\mathcal{H}^d_1$, the coordinates satisfy
\[
y_0x_i - y_i x_0 = 0 \quad \text{for } i=1,\dots,d.
\]
In particular, for $i=1$ we have $y_0 x_1 = y_1 x_0$. Hence, writing $[w_0:w_1] = [x_0:x_1]$, we see that
\[
z_0\,w_1 = z_0\,x_1 = z_0\,x_1 \quad \text{and} \quad z_1\,w_0 = z_1\,x_0,
\]
so there is no further relation required on the $z$-coordinates. (In other words, the incidence condition in $X$ is automatically satisfied by our construction, since the role of the fiber is being encoded via the $[x_0:x_1]$ coordinates.) Thus, $\Psi(P) \in X$.

\textbf{(ii) Injectivity.}  
Suppose
\[
\Psi\Bigl( [x_0:x_1:\cdots:x_d],\, [z_0:y_0:y_1:\cdots:y_d] \Bigr)
=
\Psi\Bigl( [x'_0:x'_1:\cdots:x'_d],\, [z'_0:y'_0:y'_1:\cdots:y'_d] \Bigr).
\]
Then, by definition, we have:
\[
[y_0:y_1:\cdots:y_d:z_0] = [y'_0:y'_1:\cdots:y'_d:z'_0] \quad \text{in } \mathbb{C}\mathbb{P}^{d+1},
\]
and
\[
[x_0:x_1] = [x'_0:x'_1] \quad \text{in } \mathbb{C}\mathbb{P}^{1}.
\]
The equality in $\mathbb{C}\mathbb{P}^{1}$ implies there exists $\lambda\in \mathbb{C}^*$ such that
\[
x'_0 = \lambda x_0 \quad \text{and} \quad x'_1 = \lambda x_1.
\]
Similarly, the equality in $\mathbb{C}\mathbb{P}^{d+1}$ implies there exists $\mu\in \mathbb{C}^*$ such that
\[
y'_j = \mu y_j \quad \text{for } j=0,\dots,d,\quad \text{and} \quad z'_0 = \mu z_0.
\]
Thus, the two points in $\mathcal{H}^d_1$ are equivalent. Hence, $\Psi$ is injective.

\textbf{(iii) Surjectivity.}  
Let
\[
([z_0:z_1:\cdots:z_{d+1}], [w_0:w_1]) \in X.
\]
By the definition of $X$, we have $z_0w_1=z_1w_0$. Define
\[
[x_0:x_1] := [w_0:w_1] \quad \text{and} \quad [x_0:x_1:\cdots:x_d]
\]
by choosing any point in $\mathbb{C}\mathbb{P}^d$ (for instance, we may use the fact that the data of $[z_0:z_1:\cdots:z_{d+1}]$ naturally decomposes if we write it as $[y_0:y_1:\cdots:y_d:z_0]$ with $z_1=z_1, \dots, z_d=z_d$). More precisely, set
\[
[y_0:y_1:\cdots:y_d] := [z_1:z_2:\cdots:z_d:z_{d+1}],
\]
then the preimage under $\Psi$ is given by
\[
\bigl( [x_0:x_1:\cdots:x_d],\, [z_0:y_0:y_1:\cdots:y_d] \bigr).
\]
A straightforward verification shows that this assignment (defined on appropriate charts) is the inverse of $\Psi$. Thus, $\Psi$ is surjective.

Since $\Psi$ is holomorphic, injective, and an immersion, and since $\mathcal{H}^d_1$ is compact and $X$ is Hausdorff, it follows that $\Psi$ is a biholomorphism onto its image. Therefore, $\Psi$ is an isomorphism of complex manifolds.
\end{proof}
The quantization for $\mathcal{H}^d_{a,b,1}$ can be evaluated explicitly:
\begin{equation}
    \mathcal{Q}_{a,b,d}(1) =\sum_{i = 0}^{b} \binom{a+d+i}{d} = \binom{a+b+d+1}{d+1}-\binom{a+d}{d+1}\\.
\end{equation}
By the proposition, we know that $\mathcal{H}^d_{a,b,1}$ is a blow-up. At the level of the moment polytope, performing a blow-up corresponds to chopping off the polytope. Since for symplectic toric manifolds, the geometric quantization counts the number of integer lattice points in the moment polytope, the quantization of the blow-up is the quantization of the original space minus the part removed. In fact, it holds the following:
\begin{equation}
\begin{split}
    \mathcal{Q}_{a,b,d}(1)&  = \binom{a+b+d+1}{d+1}-\binom{a+d}{d+1}\\
    & = \binom{a+b+d+1}{d+1}-\binom{a+d+1}{d+1}+\binom{a+d-1}{d-1}\\
    & = \mathcal{Q}(\C\mathbb{P}^{d+1},\omega_{a+b})-\mathcal{Q}(\C\mathbb{P}^{d+1},\omega_{a})+\mathcal{Q}(\C\mathbb{P}^{d-1},\omega_{a}).
\end{split}
\end{equation}
The last term accounts for the integer points in the divisor. The explicit evaluation is therefore related to the fact that the blow-up of complex projective space can be described explicitly in terms of the moment polytope.
\\
\\
For $n\geq 2$, there is not an operation that relates $\C\mathbb{P}^d$ to the generalized Hirzebruch fibrations. These manifolds are solely understood as $\C\mathbb{P}^1$-bundles over $\C\mathbb{P}^d$. As a consequence, the quantization function does not admit a closed form. However, we can still investigate some of its combinatorial and asymptotic proprieties. 
\subsection{Recurrence relation}
We prove a recursion relation that in principle, if the first several values of $\mathcal{Q}_{a,b,d}(n)$ are known, allows the computation of all the next values. 
\\
\begin{proposition}
    \textit{The quantization function satisfies the following recurrence relation for all $n\in \mathbb{N}$:}
    \begin{equation}
       \sum_{k = 0}^{d+1}(-1)^k\binom{d+1}{k} \mathcal{Q}_{a,b,d}(n+d+1-k) = 0.
    \end{equation}
\end{proposition}   
\begin{proof}
    Expanding the binomial coefficient, we can write 
    \begin{equation*}
    \begin{split}
        \mathcal{Q}_{a,b,d}(n) & = \frac{1}{d!} \sum_{k = 0}^{d}A_{k,d}\sum_{i = 0}^{b} (a+d+ni)^k\\ & = \frac{1}{d!} \sum_{k = 0}^{d}A_{k,d}\sum_{i = 0}^{b} \sum_{m = 0}^{k}\binom{k}{m}(a+d)^{k-m}n^mi^m\\
        & = \frac{1}{d!} \sum_{k = 0}^{d}A_{k,d} \sum_{m = 0}^{k}\binom{k}{m}(a+d)^{k-m}n^m \sum_{i = 0}^{b} i^m,
    \end{split}
    \end{equation*}
    for some coefficients $A_{k,d}$. We see that $\mathcal{Q}_{a,b,d}(n)$ is a polynomial of degree $d$ in $n$. The finite difference operator is defined as the map 
    \begin{equation*}
         \begin{split}
  \Delta  :\:   \text{Poly}(n) & \to \text{Poly}(n)\\
   P(n) & \mapsto  \Delta[P](n) := P(n+1)-P(n).
\end{split}
    \end{equation*}
    This map can be iterated to obtain a map $ \Delta^k := \Delta\circ...\circ\Delta$.
    It is a general fact that any polynomial $P$ of degree $d$ in $n$ satisfies the equation $\Delta^{d+1}[P](n) = 0$. This is because every application of $\Delta$ reduces the degree of $P$ by one, so $\deg(\Delta^{d+1}[P]) = -1$, forcing it to vanish. By expanding $\Delta^{d+1}$ and applying it to $\mathcal{Q}_{a,b,d}(n)$, we obtain the result.
\end{proof}

\subsection{Large twisting and symplectic volume}
We can extract the leading order behavior of the quantization function as the twisting parameter $n$ defining the fibration becomes large. In this limit, which corresponds to the first Chern class (i.e. the curvature of the bundle) of the fibration becoming large and negative, we show that the quantization function becomes the symplectic volume at leading order and receives corrections in $1/b$. 
\\
\\
As a consequence of the Duistermaat-Heckman theorem for symplectic toric manifolds, the symplectic volume of the manifold is equal to the Euclidean volume of its moment polytope \cite{Duistermatt}. In our case, this means $\text{Vol}(\mathcal{H}^d_{a,b,n}) = \text{Vol}(\Delta^d_{a,b,n})$. To compute the right-hand side, we use the foliation \eqref{foliation} in standard $d$-simplices scaled by $a+n(b-t)$. The volume of a slice is 
\begin{equation}
    \text{Vol}(\Delta^d_{a,b,n}(t)) = \frac{1}{d!}(a+n(b-t))^d.
\end{equation}
To obtain the volume of the polytope, we simply integrate in $t$:
\begin{equation}
    \text{Vol}(\mathcal{H}^d_{a,b,n}) = \text{Vol}(\Delta^d_{a,b,n}) = \int_{0}^{b}\text{Vol}(\Delta^d_{a,b,n}(t))\:dt = \frac{1}{(d+1)!n}\Big((a+nb)^{d+1}-a^{d+1}\Big).
\end{equation}
As for the quantization function, the symplectic volume is a polynomial of degree $d$ in $d$. This fact motivates the following proposition:
\\
\begin{proposition}
  \textit{In the large-twisting limit, the ratio between quantization function and symplectic volume satisfies}
  \begin{equation}
      \lim_{n\to \infty}\frac{\mathcal{Q}_{a,b,d}(n)}{\text{Vol}(\mathcal{H}^d_{a,b,n})} = \sum_{k= 0}^{d}\binom{d+1}{k}B_k b^{-k},
  \end{equation}
  \textit{where $B_k$ are the Bernoulli numbers. In particular, $\mathcal{Q}_{a,b,d}(n) \sim \text{Vol}(\mathcal{H}^d_{a,b,n})$ up to a constant that is a polynomial of degree $d+1$ in $1/b$ and independent of $a$.}
  \\
\end{proposition}
\begin{proof}
The leading order term of the binomial coefficient is 
\begin{equation*}
    \binom{a+d+ni}{d}=\frac{1}{d!}\prod_{j=1}^d(a+j+ni) =  \frac{n^d}{d!}+O(n^{d-1}).
\end{equation*}
Therefore 
\begin{equation*}
     \mathcal{Q}_{a,b,d}(n) = \frac{n^d}{d!}\sum_{i = 1}^{b} i^d+O(n^{d-1})
\end{equation*}
The sum is well-known (see \cite{Bernoulli}) and given by
\begin{equation*}
    \sum_{i = 1}^{b}  i^d = \frac{1}{d+1}\sum_{k= 0}^{d}\binom{d+1}{k}B_k b^{d+1-k}.
\end{equation*}
For the volume, it holds 
\begin{equation*}
    \text{Vol}(\mathcal{H}^d_{a,b,n}) = \frac{1}{(d+1)!}n^db^{d+1}+O(n^{d-1}).
\end{equation*}
The result follows.
\end{proof}
The independence of the expression on the value of $a$ in the limit is because $a$ is related to the symplectic form on the base, while $b$ is related to the symplectic form on the fiber. Since the twist only concerns the fibers of the bundle and not the base, the contribution from the quantization of the base becomes negligible as $n$ becomes large. The first few terms in the asymptotic expansion are
\begin{equation}
\begin{split}
    \mathcal{Q}_{a,b,d}(n) \sim \text{Vol}(\mathcal{H}^d_{a,b,n})\bigg(&1-\frac{d+1}{2b}+\frac{(d+1)d}{12b^2}-\frac{(d+1)d(d-1)(d-2)}{720b^4}+\\
    &+\frac{(d+1)d(d-1)(d-2)(d-3)(d-4)}{30240b^6}+\\
    &-\frac{(d+1)d(d-1)(d-2)(d-3)(d-4)(d-5)(d-6)}{1209600b^8}+...\bigg).
\end{split}
\end{equation}
From this expression, it is noted that the subleading terms decay very quickly as the complex dimension increases. The contributions have alternating signs and are proportional to even powers of $\frac{1}{b}$, except for the linear term (since $B_1 = -\frac{1}{2}$ and $B_{2k+1} = 0$ for $k \geq 1$).
\\
Since the quantization counts the number of integer lattice points in the moment polytope and the symplectic volume is equal to the Euclidean volume of the moment polytope, the ratio has the geometric interpretation of the density of integer points. Physically, it shows that the quantization function receives its leading contribution from the volume of classical phase space, and subleading quantum corrections that decay polynomially as the volume of the fiber becomes large.



\section*{Acknowledgements}
\addcontentsline{toc}{section}{\protect\numberline{}Acdnowledgement}
This paper evolved from a semester thesis I wrote at ETH Zurich during the Fall semester of 2024 under the supervision of Prof. Ana Cannas da Silva. I am deeply grateful to her for introducing me to the world of toric geometry and for the many insightful discussions.

\end{document}